\documentclass[10pt]{amsart}
\usepackage{amssymb}
\usepackage[all]{xy}
\usepackage{eurosym}

\newcommand{\Ph}{\operatorname{\bbP h}}
\newcommand{\bPh}{\operatorname{\overline{\bbP h}}}
\newcommand{\PI}{\operatorname{\bbP I}}

\newcommand{\Sh}{\operatorname{\bbS h}}
\newcommand{\bSL}{\operatorname{\overline{\SL}}}
\newcommand{\bA}{\overline{A}}
\newcommand{\bO}{\overline{O}}
\newcommand{\bF}{\overline{F}}

\newcommand\DP{\operatorname{CP}}
\newcommand\bSH{\operatorname{\overline{\mathbb{S}H}}}
\newcommand\SH{\operatorname{\mathbb{S}H}}
\newcommand\SIp{\operatorname{\mathbb{S}I_+}}
\newcommand\SI{\operatorname{\mathbb{S}I}}

\newcommand\bq{\bar q}

\newcommand\bM{\overline{M}}
\newcommand\bG{\overline{G}}

\newcommand\diag{\operatorname{diag}}

\newcommand\ilog{\operatorname{ilog}}

\newcommand\mrurl[1]{MR. \url{http://www.ams.org/mathscinet-getitem?mr=#1}}
\newcommand\aurl[1]{\url{file:///data/ms/archive/#1}}
\newcommand\durl[1]{doi. \url{http:///dx.doi.org/#1}}
\newcommand\purl[1]{\url{file:///data/ms/print/#1}}
\newcommand\surl[1]{\url{file:///data/ms/screen/#1}}
\newcommand\zurl[1]{Zbl. \url{http://www.emis.de/zmath-item?#1}}
\newcommand\Aurl[1]{Arxiv. \url{http://arxiv.org/abs/#1}}

\newcommand\boxb[1]{\square_b}

\numberwithin{equation}{section}

\newcommand\paperbody%
        {}
\newcommand\appendixbody%
{\appendix\setcounter{equation}{0}
        }

\newtheorem{lemma}{Lemma}
\newtheorem{proposition}{Proposition}

\newtheorem{theorem}{Theorem}
\newtheorem{non-theorem}{Non-Theorem}

\newtheorem*{conj}{Conjecture}

\newtheoremstyle{named}{}{}{\itshape}{}{\bfseries}{.}{.5em}{\thmnote{#3 }#1}
\theoremstyle{named}
\newtheorem*{namedtheorem}{Theorem}
\newtheoremstyle{fullnamed}{}{}{\itshape}{}{\bfseries}{.}{.5em}{\thmnote{#3}}
\theoremstyle{fullnamed}
\newtheorem*{fullnamed}{}
\theoremstyle{remark}
\newtheorem{definition}{Definition}

\newcommand\bV{\mathcal{V}_{\operatorname{b}}}
\newcommand\cV{\mathcal{V}_{\operatorname{c}}}

\newcommand\bo{\operatorname{b}}

\newcommand\Ad{\operatorname{Ad}}

\newcommand\bT{{}^{\bo}T}

\newcommand\cFTs{{}^{\Phi}\overline{T}\kern-1pt{}^*}

\newcommand\bN{{}^{\bo}N}

\newcommand\Tr{\operatorname{Tr}}

\newcommand\SO{\operatorname{SO}}
\newcommand\GL{\operatorname{GL}}

\newcommand\SL{\operatorname{SL}}

\newcommand\Fg{\mathfrak{g}}
\newcommand\Fm{\mathfrak{m}}
\newcommand\Fn{\mathfrak{n}}
\newcommand\Fk{\mathfrak{k}}

\newcommand\cF{\mathcal F}

\renewcommand\cV{\mathcal{V}}

\newcommand\dcF{{}^2\kern-1.5pt\mathcal{F}}
\newcommand\DcA{{}^2\kern-3pt\mathcal{A}}
\newcommand\DcF{{}^2\kern-1.5pt\mathcal{F}}
\newcommand\DcL{{}^2\kern-1.5pt\mathcal{L}}

\newcommand\bbC{\mathbb C}

\newcommand\bbN{\mathbb N}
\newcommand\bbP{\mathbb P}

\newcommand\bbR{\mathbb R}
\newcommand\bbS{\mathbb S}
\newcommand\bbT{\mathbb T}

\newcommand\bbZ{\mathbb Z}

\newcommand\CI{{\mathcal{C}}^{\infty}}

\newcommand\cFNs{{}^{\Phi}\overline N\kern-1pt{}^*}

\newcommand\Gr{\operatorname{Gr}}

\newcommand\Hom{\operatorname{Hom}}

\newcommand\SU{\operatorname{SU}}

\newcommand\UU{\operatorname{U}}

\newcommand\ci{${\mathcal{C}}^\infty$}

\newcommand\ha{\frac{1}{2}}

\newcommand\pa{\partial}

\newcommand\Fa{\mathfrak{a}}

\newcommand\Mif{\text{ if }}

\newcommand\Mor{\text{ or }}

\begin{document}
\title[CoSLG]
{Compactification of semi-simple Lie groups}

\author{Pierre Albin}
\author{Panagiotis Dimakis}
\author{Richard Melrose}
\address{Department of Mathematics, Massachusetts Institute of Technology}
\email{rbm@math.mit.edu}
\author{David Vogan}
\address{Department of Mathematics, Massachusetts Institute of Technology}
\email{dav@math.mit.edu}

\begin{abstract} We discuss the `hd-compactification' of a semi-simple Lie
  group to a manifold with corners; it is the real analog of the wonderful
  compactification of deConcini and Procesi. There is a 1-1 correspondence
  between the boundary faces of the compactification and conjugacy classes
  of parabolic subgroups with the boundary face fibering over two copies of
  the corresponding flag variety with fiber modeled on the
  (compactification of the) reductive part.  On the hd-compactification
  Harish-Chandra's Schwartz space is identified with a space of conormal
  functions of rapid-logarithmic decay relative to square-integrable
  functions.
\end{abstract}

\maketitle


\section*{Introduction}

In this note we describe a systematic compactification procedure for Lie
groups. Although semisimple Lie groups are the primary focus of interest
for well known `iterative' reasons we work in the setting of real
reductive groups with compact centers.  The `hd-compactification' of such a
group, $G,$ is a compact manifold with corners, $\bG,$ with interior
identified with $G$ and with additional properties described below. The
hd-compactification is closely related to (derivable from) the wonderful
compactification of De Concini and Procesi, \cite{MR718125}, in the case of
complex adjoint groups, and to the (maximal) Satake compactification, and
more especially the Oshima compactification, \cite{Oshima}, of the
homogeneous space $G/K$ for a maximal compact subgroup $K;$ see also \cite{GJT}.

By a compactification $\bM$ of a manifold $M$ without boundary
we mean a smooth map $I:M\hookrightarrow \bM$ which is a diffeomorphism
onto the interior of a compact manifold with corners $\bM.$ Two
compactifications are \emph{equivalent} if there is a diffeomorphism
between them intertwining the inclusions into the interiors. As part of
the definition we always demand that the boundary hypersurfaces of a
compact manifold (by default meaning with corners) are embedded. This
implies that each such hypersurface $H$ has a global boundary defining
function $\rho_H\in\CI(M)$ such that $\rho _H>0$ on $\bM\setminus H,$
$H=\{\rho =0\}$ and $d\rho_H\not=0$ at $H.$ Thus $d\rho_H\big|_H$ spans the
conormal bundle to $H.$ These assumptions also imply that the components of
the intersections of the hypersurfaces, the boundary faces, are all
embedded and are naturally compact manifolds with corners.

In the category of manifolds with corners the arrows are required to be
smooth maps $f:X'\longrightarrow X''$ in the usual sense that
$f^*\CI(X'')\subset\CI(X')$ but also they are b-maps meaning that for
each boundary hypersurface $H$ of $X''$
\begin{equation*}
f^*\rho_H=%
\begin{cases}
\equiv0&\Mor\\
a\prod_{K}\rho_K^{n_{KH}},\ 0<a\in\CI(X')&\Mor\\
>0.
\end{cases}
\label{CSG.146}\end{equation*}
The maps we are most interested in are interior b-maps, in which the first
option does not occur. The powers, $n_{KH},$ in the product decomposition
over the boundary hypersurfaces of $X'$ are necessarily non-negative
integers (so the last case is where these integers all vanish).

The Lie algebra of the group of diffeomorphisms of a compact manifold
consists of the smooth vector fields tangent to all boundary faces, the
b-vector fields. Equivalently these are the smooth vector fields satisfying
$V\rho _H\in\rho_H\CI(X)$ for all boundary defining functions $\rho _H.$
The b-vector fields form a Lie algebroid, $\bV(X)=\CI(X;\bT X),$ where $\bT
X\longrightarrow TX$ has null space at each boundary point $p$ of
codimension $k$ i.e.\ lying in the interior of a boundary face $F$ of
codimension $k,$ a canonically trivial vector space $\bN_pF\subset\bT M$
which extends to a smooth subbundle over $F.$ These spaces are spanned by
the vector fields $x_i\pa_{x_i}$ in terms of local coordinates in which the
$x_i$ define the boundary hypersurfaces through $p.$

\begin{definition}\label{CSG.167} An hd-compactification of a real reductive
Lie group with compact center is a compact manifold with corners $\bG$ and a
diffeomorphism onto the interior $G\hookrightarrow \bG$ such that
\begin{enumerate}
\item[(D1)]\label{D.1}[inversion] Inversion extends to a diffeomorphism of $\bG.$ 
\item[(D2)]\label{D.2}[b-normality] The right action of $G$ extends smoothly to
  $\bG$ with isotropy algebra at each boundary point projecting to span the
  b-normal space.
\item[(D3)]\label{D.3}[b-transitivity] The combined action of $G\times G$ on left
  and right is b-transitive, i.e.\ has Lie algebra spanning $\bT\bG.$
\item[(D4)]\label{D.4}[minimality]
Near the interior of a boundary face of codimension $d$
  the span of the Lie algebra for the right action contains vector fields
  $x _jv_j$ where the $x_j$ are defining functions for the local
  boundary hypersurfaces and the $v_j$ are locally independent tangent
  vector fields not themselves in the span of the Lie algebra.
\end{enumerate}
\end{definition}

\begin{namedtheorem}[Main]\label{CSG.165} Any real
  reductive group with compact center has an hd-compact\-ification
\end{namedtheorem}

\begin{conj}\label{CSG.222}  The hd-compactification of a
  real-reductive Lie group with compact center is
  unique up to equivalence, intertwining the right and left actions, and so
  defines a functor from these groups and isomorphisms to compact manifolds
  with corners and diffeomorphisms.
\end{conj}

\noindent In view of the invariance shown below this hd-compactification
extends to spaces which are left and right principal $G$ spaces, reduced to
$G$ by fixing a point, for a real reductive group $G.$

One direct consequence of the properties demanded above is that the action,
on left or right, of a maximal compact subgroup $K\subset G$ on $\bG$ is
necessarily free. Thus $\bG/K$ defines a compactification, essentially the
maximal Satake compactification, of $G/K,$ see Borel and Ji
\cite{MR2189882}.

Below we give a detailed differential-geometric description of the
hd-compact\-ification with subsequent geometric analysis in mind, see
Mazzeo an Vasy \cite{MR2175410}, \cite{MR2075483} and Parthasarathy and Ramacher
\cite{MR3217053}. In particular the boundary faces of $\bG$ are shown to be
in 1-1 correspondence with the conjugacy classes of parabolic subgroups of
$G$ and hence with the subsets $S\subset D$ of the nodes of the reduced
Dynkin diagram. If $\cF_S$ is the flag variety parameterizing the
parabolics associated to $S$ then the corresponding boundary face of $\bG$
can be realized as a bundle
\begin{equation}
\xymatrix{
\bM\ar@{-}[r]&\bF_S\ar[d]\\
&\cF_S\times\cF_S
}
\label{CSG.161}\end{equation}
where the model fibre is the compactification of the real-reductive group
in the Langlands decomposition $\cF_S\ni P=M_PA_PN_P$ of an associated parabolic. More
geometrically the fibre at $(P',P)$ is the hd-compactification of the space
of elements of $G$ such that $M_{P'}g=gM_P.$ The
isotropy group for the right action of $G$ at a point of $F_S$ is the
associated normal subgroup $A_PN_P.$

We show how to assemble a compactification
\begin{equation}
\bG=G\sqcup\bigcup_{S\subsetneq D}F_S
\label{CSG.162}\end{equation}
from the unions of the interiors of the boundary faces. To do so
we define `gluing maps' from the putative normal bundles to the $F_S$
which have fiber at a point of $F_S$ a partial compactification of the
corresponding group $A_P$ and show that this is an
hd-compactification.

For analytic purposes it is of prime importance to describe the behavior of
the invariant vector fields for the action of $G$ on $\bG;$ this can be
seen from the root space decomposition as discussed by Knapp
\cite{MR1880691}, Wallach \cite{MR929683} and Varadarajan
\cite{MR1725738}. To this end we observe that the flag varieties carry
iterated tangent structures. If $P\in\cF_S$ is a parabolic subgroup in a
fixed conjugacy class then the transitive action of $K,$ by conjugation, on
$\cF_S$ has isotropy group $K\cap P$ at $P.$ In the Langlands decomposition
$P=M_PA_PN_P$ the Lie algebra $\Fa_P$ of $A_P$ has rank $s=\#(D\setminus
S).$ The positive root vectors, joint eigenvectors for the conjugation
action of $\Fa_P$ on the Lie algebra span the Lie algebra of $\Fn_P.$ The
eigen-decomposition of $\Fk$ defines a surjective map $\Fk\longrightarrow
\Fn_P$ with null space $\Fk\cap\Fm_P,$ the Lie algebra of the isotropy
group. The decomposition of $\Fn_P$ into eigenspaces with joint eigenvalues
$\alpha \cdot\Fa_P$ thus induces a corresponding filtration
\begin{equation}
E^\alpha _P\subset T_P\cF_S.
\label{CSG.202}\end{equation}
The action of $K$ extends the filtration to $T\cF_S$ by subbundles
labeled by $s$-multi\-indices satisfying
\begin{equation}
\begin{gathered}
E^\alpha \subset T\cF_S\ \forall\ \alpha \in \bbN^s,\ E^\alpha
=T\cF_S\text{ for some }\alpha \\
\alpha \le\beta\Longrightarrow E^\alpha\subset E^\beta,\\
E^\alpha +E^\beta \subset E^{\alpha+\beta},\\
[\cV_\alpha ,\cV_\beta ]\subset\cV_{\alpha +\beta },\ \cV_\alpha =\CI(\cF_S;E^\alpha).
\end{gathered}
\label{CSG.163}\end{equation}

The iterated structures on intersecting boundary faces are related in such
a way that they induce a Lie algebroid structure $E_R\subset \bT\bG$
consisting of the vector fields tangent to all boundary faces, to the
fibers of the left fibration to $\cF_S$ induced by \eqref{CSG.161} for each
$S$ and with normal vanishing properties corresponding to the $E^\alpha.$

As noted above there is a direct relationship between the wonderful
compactification of the adjoint form of a complex semisimple group and the
hd-compactification.

\begin{fullnamed}[Prop-W]\label{CSG.166} The real blow up of the exceptional divisors
  in the wonderful compactification of the adjoint group of a complex
  semisimple group is an hd-compactification.
\end{fullnamed}

At the boundary of the hd-compactification Haar measure takes the form 
\begin{equation}
dg=\rho^{-\sigma}_*\nu_{\bo}
\label{CSG.169}\end{equation}
where $\nu_{\bo}$ is a non-vanishing smooth b-measure, $\rho_*$ is a
vector of defining functions for the boundary hypersurfaces (corresponding to maximal
parabolic subgroups, so to the sets $D\setminus\{n\})$ for $n\in D)$  and $\sigma$ is
the multiindex given by the sum of the positive roots.

As already noted in \cite{1812.03883} the utility of the
hd-compactification is illustrated by the fact that Harish-Chandra's Schwartz
space takes a relatively simple form with respect to the
hd-compactification.

\begin{fullnamed}[Prop-H-C]\label{CSG.170} Harish-Chandra's Schwartz space is the
  space of conormal functions with respect to the boundary of $\bG$ of
  log-rapid decay with respect to $L^2(G).$
\end{fullnamed}
\noindent This was proved for $\SL(2,\bbR)$ in \cite{1812.03883} but the argument is sufficiently
general to apply here.

In \S\ref{GA} general properties of groups actions on manifolds with
corners are discussed and the characterization of Harish-Chandra's Schwartz
space is recalled from \cite{1812.03883}. The special case of $\SL(2,\bbR)$
is treated in \S\ref{nR} where the hd-compactification is obtained by
blow-up of the sphere in the homomorphisms. For the closely related case of
$\SL(n,\bbC)$ both the wonderful and the hd-compactification are obtained
by similar methods in \S\ref{nC}. The construction in the general case is
contained in \S\ref{DC}.

\paperbody

\section{Group actions on manifolds with corners}\label{GA}

First consider the smooth action of a compact Lie group on a compact
manifold with corners; the fact that the boundaries are `one-sided' imposes
constraints on such actions since the isotropy algebra at boundary points
must act trivially on the normal bundle to the boundary.

\begin{lemma}\label{CSG.157} If a compact Lie group acts smoothly on a
 connected compact manifold with corners and acts trivially on one boundary
 face then it acts trivially.
\end{lemma}

\begin{proof} This is a consequence of linearizability of compact group
  actions. Linearizing the action around a point in the interior of a
  boundary hypersurface shows that the local action is the same as that on
  the spherical normal bundle. If the action on the hypersurface is trivial
  then, from the one-sided nature of the boundary, there is no $\bbZ_2$
  action between the two normal directions so the action on the normal
  bundle is also trivial and hence the action is trivial nearby. Again the
  triviality of the action on an open set implies that the action is
  trivial on that component of the space. Applying this argument
  iteratively, triviality on any one boundary face implies triviality on
  all faces of which it is a boundary hypersurface, and hence triviality on
  the full manifold.
\end{proof}

In particular a smooth action of a connected compact group on a connected
compact manifold with corners up the dimension, so including points, is
necessarily trivial. This includes simplexes.

Conversely there is a similar inheritance of freedom of an action.

\begin{proposition}\label{CSG.158} If a compact group acts smoothly on a
compact manifold with corners with the action free on the interior then it
acts freely.
\end{proposition}

\begin{proof} Consider the isotropy group at any boundary point, which is
  always in the interior of some boundary face. By linearizability, the fact
  that the action is free in the interior implies that the isotropy group
  must act freely and linearly on the fibres of the inward-pointing part of
  the spherical normal bundle; since this is a simplex the action must
  preserve all boundary faces and so is trivial and hence so is the
  isotropy group and the full action must be free.
\end{proof}

One of the most fundamental requirements that we place on the
compactification of a group $G\hookrightarrow \bG$ as the interior of a
compact manifold with corners in Definition~\ref{CSG.167} is that the left
and right actions of $G$ extend smoothly to $\bG.$ The existence of such a
smooth extension of the action is equivalent to the condition that the
left- and right-invariant vector fields extend from the interior to be
b-vector fields on $\bG.$ Thus each action gives a Lie algebra homomorphism
\begin{equation}
  \Fg\longrightarrow \bV(\bG)=\CI(\bG;\bT).
\label{CSG.215}\end{equation}

A compactification for which the actions extends smoothly is by no means
unique since one can make `transcendental' changes to $\bG$ under which the
smooth action of the vector fields lifts. Suppose one has such a (weak)
compactification. Take a defining function for one of the boundary
hypersurfaces, $x,$ and replace it by
\begin{equation}
t=\ilog x=\frac{1}{\log{\frac1x}},\ x=\exp(-\frac1t).
\label{CSG.148}\end{equation}
The manifold with new \ci\ structure generated near this boundary
hypersurface by $t,$ in addition to $x$ and all other smooth functions, is
homeomorphic to the original $\bG$ and maps smoothly back to it.  Moreover,
a simple computation shows that the tangential vector fields lift to be
smooth and tangential on the manifold, since $x\pa_x$ lifts to $t^2\pa_t.$
Thus the `blown up' manifold is still a weak compactification but is, in a
sense, much larger and certainly not equivalent.

The condition \eqref{D.3} in the definition of an hd-compactification in
the Introduction prevents this extreme sort of `enlargement'. 

\begin{definition}\label{CSG.149} The action of a Lie group on a compact
  manifold with corners is \emph{b-normal} if at each boundary point the
  b-normal space is spanned by elements of the Lie algebra of the isotropy
  group.
\end{definition}

In fact this still does not prevent `blow-up' indeterminacy, as discussed
above, from arising. However, the transformations can no longer be
transcendental. Still, the introduction of $s=x^{\tfrac1k}$ in place of $x$
at a boundary hypersurface, for any positive integer $k,$ again defines a
smooth map from the `resolved' space which is a homeomorphism and under
which the b-vector fields lift to smooth b-vector fields -- so a smooth
action lifts to be smooth. In this case the b-normal $x\pa_x$ is replaced
by $k^{-1}t\pa_t.$ Thus b-normality is also preserved. In consequence we
need to impose the further `minimality' condition, \eqref{D.4}, in the
definition of an hd-compactification, to ensure uniqueness. 

The action of $G$ on itself is of course transitive hence so is the
combined right and left action of $G\times G.$ In \eqref{D.3} we impose a
condition which, in a weakened sense, extends this to the compactification.

\begin{definition}\label{CSG.153} The action of a Lie group on a manifold
  with corners is \emph{b-transitive} if its Lie algebra generates the b-vector
  fields as a \ci\ module.

A \emph{b-compactification} $G\hookrightarrow \bG$ is an inclusion as the interior
of a compact manifold with corners such that the right and left actions
extend smoothly and are b-normal and the combined action $G\times G$ is
b-transitive.
\end{definition}

Since the interiors of all boundary faces are (by assumption) connected,
a b-transitive action is transitive on these interiors.

\begin{lemma}\label{CSG.159} For a b-compactification the isotropy groups
  of the left and right actions at the interior are conjugate over the
  interior of any boundary face.
\end{lemma}

\begin{proof} Since the right and left actions commute the orbits
  through points in the interior of any one boundary face are
  equivariantly diffeomorphic. 
\end{proof}

The proof of the analogous result in \cite{1812.03883} only uses properties
holding for a b-compactification and \eqref{CSG.169} so yields Prop-H-C of
the Introduction.

\section{$\SL(n,\bbR)$}\label{nR}

Before proceeding to the construction of an hd-compactification in general
we consider the case of $\SL(n,\bbR)$ (and subsequently $\SL(n,\bbC)).$ For
these standard groups we give an explicit construction of an
hd-compactification which illustrates and guides the more abstract construction
below. For later convenience in place of $\bbR^n$ consider an oriented real
Euclidean vector space, $V,$ of dimension $n.$ Then consider the inclusions
\begin{equation}
\SO(V)\subset\SL(V)\subset\GL(V)\subset\Hom(V),
\label{CSG.171}\end{equation}
where $\Hom(V)=\Hom(V,V)$ is the space of linear maps. We give $\Hom(V)$
the usual Hilbert-Schmidt norm $\|e\|=\Tr(e^*e)^{\ha}.$

Let $\SH(V)\subset\Hom(V)$ be the unit sphere and consider the map given by
(positive) radial scaling
\begin{equation}
\SL(V)\longrightarrow \SI(V)\subset\SH(V)
\label{CSG.113}\end{equation}
where $\SI(V)=\SH(V)\cap\GL(V)$ is the open subset consisting of the invertible
homomorphisms of norm one.

\begin{lemma}\label{CoCLG.114} Radial scaling \eqref{CSG.113} is a
  diffeomorphism onto its range $\SIp(V)\subset\SI(V),$ the open subset of
  $\SH(V)$ where the determinant is positive.
\end{lemma}

\begin{proof} The inverse of \eqref{CSG.113} is 
\begin{equation}
\SIp(V)\ni e\longrightarrow e/(\det(e))^{\tfrac1n}\in\SL(V)
\label{CSG.115}\end{equation}
so this map is a diffeomorphism.
\end{proof}

The hd-compactification of $\SL(V)$ is obtained by compactifying the range
$\SIp(V)$ of \eqref{CSG.113} and this in turn is accomplished through
blow-up of $\SH(V).$ Again for use in inductive arguments below we proceed
slightly more generally. Suppose $W_1$ and $W_2$ are two oriented real
Euclidean vector spaces of the same dimension, $k.$ Then consider the unit
sphere (always in the Hilbert-Schmidt norm)
$\SH(W_1,W_2)\subset\Hom(W_1,W_2).$ The orthogonal groups $\SO(W_i)$ act on
$\SH(W_1,W_2)$ on the right and left and both these actions are free on
$\SI(W_1,W_2)\subset\SH(W_1,W_2),$ the open set of invertible
homomorphisms. Following \cite{MR2560748} these actions may be resolved by
successive blow-up of the isotropy types which are the
\begin{equation}
S_q(W_1,W_2)=\{e\in\SH(W_1,W_2);e\text{ has corank }q\},\ 1\le q\le k-1;
\label{CSG.127}\end{equation}
here $q$ is the `depth' of the stratum.

Denote the resolution, (obtained by blow up in order of increasing
dimension or decreasing depth),
\begin{equation}
\bSH(W_1,W_2)=[\SH(W_1,W_2),S_*].
\label{CSG.129}\end{equation}
In case $W_1=W_2=W$ we set $\bSH(W,W)=\bSH(W)$ and for $W=\bbR^k,$ $\bSH(W)=\bSH(k),$ etc.

Diffeomorphisms fixing the center lift under blow-up so the orthogonal actions lift and 
Euclidean isomorphisms $I_i:W_i\longrightarrow \bbC^k$ $i=1,2$ result in a
commutative square covering the blow-down maps 
\begin{equation}
\xymatrix{
\bSH(W_1,W_2)\ar[d]_{\beta }\ar@{<->}[r]&\bSH(k)\ar[d]^{\beta }\\
\SH(W_1,W_2)\ar@{<->}[r]&\SH(k).
}
\label{CSG.130}\end{equation}
\begin{proposition}\label{CSG.126} The resolution \eqref{CSG.129} is a compact
  manifold with corners up to codimension $k-1$ with the boundary
  hypersurface corresponding to $S_q$ fibering over the double Grassmannian
  with fiber modeled (inductively) by resolved spaces
\begin{equation}
\xymatrix{
\bSH(k-q)\times\bSH(q)\ar@{-}[r]&H_q\ar[d]\\
&\Gr(W_1,k-q)\times\Gr(W_2,k-q)
}
\label{CSG.128}\end{equation}
where the fiber above $(U_1,U_2)$ is $\bSH(U_1,U_2)\times\bSH(U_1^\perp,U_2^\perp).$
\end{proposition}

\begin{proof} Consider the deepest stratum $S_{k-1}\subset\SH(W_1,W_2)$ with
  maximal isotropy group. Each element of $S_{k-1}$ has one-dimensional
  range and $k-1$ dimensional null space. So these are determined precisely
  by the pair of lines $U_1,$ the range of the adjoint, and the
  range $U_2$ and the element itself lies in the two-point space
  $\SH(U_1,U_2)\ni e.$

Thus $S_{k-1}$ is a double cover of the product $\bbP W_2\times\bbP W_1$ of
the real projective spaces with fiber $\SH(U_1,U_2).$ By the collar
neighbourhood theorem, a neighbourhood of $S_{k-1}$ in $\SH(W_1,W_2)$
fibers over $S_{k-1}$ and can be taken to be form
\begin{equation}
(1-\|f\|^2)^{\ha}e+f,\ f\in\Hom(U_1^\perp,U_2^\perp),\ \|f\|<\epsilon.
\label{CSG.120}\end{equation}
Indeed this follows from the left and right polar decomposition of
homomorphisms close to $e.$ Both polar parts must have only one eigenvalue near
$1,$ hence with smooth eigenspace, and this gives the projection onto $S_{k-1}$ with the
remainder in $\Hom(U_1^\perp,U_2^\perp).$ This identifies the normal bundle
to $S_{k-1}$ and hence the elements of the unit normal sphere bundle, which
is the face $H'_{k-1}$ produced by blow-up of $S_{k-1}$ (before the other
blow-ups) as a fiber bundle 
\begin{equation}
\xymatrix{
\SH(k-1)\ar@{-}[r]&H_{k-1}'\ar[d]\\
&S_{k-1}
}
\label{CSG.131}\end{equation}
with fibre $\SH(U_1^\perp,U_2^\perp)$ at $(U_1,U_2).$

The naturality of blow-up shows that the subsequent steps in the resolution
of $\SH(W_1,W_2)$ restrict to $H'_{k-1}$ as the corresponding resolution of
$\SH(U_1^\perp,U_2^\perp).$ Thus the final hypersurface
$H_{k-1}\subset\bSH(W_1,W_2)$ becomes the fiber bundle 
\begin{equation}
%
\xymatrix{
\bSH(k-1)\ar@{-}[r]&H_{k-1}\ar[d]\\
&S_{k-1}
}\ \Mor\
\xymatrix{
\bSH(1)\times\bSH(k-1)\ar@{-}[r]&H_{k-1}\ar[d]\\
&\bbP W_2\times\bbP W_1
}
%
\label{CSG.172}\end{equation}
with total fiber in the second case
$\bSH(U_1,U_2)\times\bSH(U_1^\perp,U_2^\perp).$ This is \eqref{CSG.128} for $q=1.$

The structure of the other faces follows iteratively. Namely after the
blow-up of the faces $S_{k-j}$ for $j<q,$ the stratum $S_{k-q}$ in \eqref{CSG.127}
is resolved to the bundle 
\begin{equation}
\xymatrix{
\bSH(k-q)\ar@{-}[r]&S_{k-q}'\ar[d]\\
&\Gr(W_2,k-q)\times\Gr(W_1,k-q)
}
\label{CSG.173}\end{equation}
where the fibre is $\bSH(U_1,U_2)$ over the pair $(U_2,U_1)$ in
the Grassmannian. The fiber of the normal bundle is precisely
$\SH(U_1^\perp,U_2^\perp)$ by essentially the same argument in terms of the
two polar decompositions. The subsequent blow-ups resolve the fibers to
give \eqref{CSG.128}.
\end{proof}

We define a compactification of $\SL(V)$ by taking the closure of the image
of \eqref{CSG.113}: 
\begin{equation}
\bSL(V)=\text{Closure of $\SIp(V)$ in $\bSH(V).$}
\label{CSG.174}\end{equation}
In fact this is one of the two components of $\bSH(V);$  $\bSH(V)$ is an
hd-compactification of $\bbZ_2\times\SL(V)$ and from this we deduce that
\begin{proposition}\label{CSG.175} The compact manifold in
  \eqref{CSG.174} is an hd-compactification of $\SL(V)$ for an oriented
  Euclidean space $V.$
\end{proposition}

\begin{proof} The description of the iterative blow-up allows the blow-down map to be
described near a point of (exactly) codimension $r$ in $\bSH(W_1,W_2).$ This
corresponds to a subset 
\begin{equation}
\bq\subset\{1,\dots,k-1\},\ \#(\bq)=r,\ q_i\in\bq,\ 1\le
q_1<q_2<\dots<q_r\le k-1.
\label{CSG.133}\end{equation}
Then the corresponding boundary face $F_{\bq}=H_{q_1}\cap H_{q_2}\dots\cap
H_{q_r}$ fibres over the product of two copies of the flag manifold
\begin{equation}
\cF(W_j,\bq)=\{U_{j,\bq}\},\ U_{i,q_r}\subset U_{i,q_{r-1}}\subset\dots
U_{q_1}\subset W_i,\ \dim(U_{j,q_j})=k-q_j.
\label{CSG.134}\end{equation}
We set $q_{r+1}=k$ and $q_0=0$ for notational convenience and
$U_{j,q_{r+1}}=W_j,$ $U_{j,q_0}=\{0\}.$ Then the flag defines a sequence of
orthogonal projections, 
\begin{equation*}
\pi_{j,q}\text{  onto }V_{j,q_i}=U_{j,q_i}\ominus U_{j,q_{i+1}}.
\label{CSG.136}\end{equation*}
There are $r$ local defining functions $\tau_{q_i}\in[0,\epsilon)$ and
an element of $\bSH(W_1,W_2)$ projecting to this neighbourhood is of the
form 
\begin{equation}
\gamma =\sum\limits_{i=1}^r(\prod_{i'>i}\tau_{q_{i'}})
\pi_{2,q_i}e_{q_i}\pi_{1,q_i},\ e_{q_i}\in\SI(V_{1,q_i},V_{2,q_i}).
\label{CSG.135}\end{equation}

This point, after radial rescaling, lies in $\bSL(V)$ if all the
$\tau_q\ge0$ and the determinant at $\tau_q=1$ is positive. From
\eqref{CSG.115} the corresponding point in $\SL(n,\bbR),$ for all
$\tau_i>0,$ is 
\begin{equation}
g=a\left(\prod_{i=1}^r\tau_i^{-d_i/n}\right)\sum\limits_{i=1}^r(\prod_{i'>i}\tau_{q_{i'}})
\pi_{2,q_i}e_{q_i}\pi_{1,q_i}
\label{CSG.203}\end{equation}
where $a>0$ is locally smooth. The inverse is therefore of the form 
\begin{equation}
g^{-1}=a\left(\prod_{i=1}^r\tau_i^{(d_i-n)/n}\right)\sum\limits_{i=1}^r(\prod_{i'\le
  i}\tau_{q_{i'}})\pi_{1,q_i}e_{q_i}^{-1}\pi_{2,q_i}.
\label{CSG.204}\end{equation}

From this it follows that inversion is a diffeomorphism on $\bSL(n,\bbR).$
That the right and left actions extend smoothly follows from the
construction since these actions are smooth after projection to $\SH$ and
fix the centers of blow-up, so lift smoothly to $\bSL.$ 

A point in the flag variety determines a parabolic subgroup although this
involves a consistent choice of Weyl chamber. Consider the standard flag
with $k$th subspace spanned by the first $k$ elements of the standard
basis of $\bbR^n.$ This fixes an Iwasawa decomposition 
\begin{equation}
\SL(n)=\SO(n)AN
\label{CSG.216}\end{equation}
where $A$ is the subgroup of positive diagonal matrices and $N,$
corresponding to a choice of positive roots, consists of the upper diagonal
matrices with all diagonal entries one. Thus the Lie algebra of $N$ is
spanned by the elementary matrices $E_{ij}$ with one non-zero entry, $1$ at
$(i,j)$ for $j>i.$ The Lie algebra, $\Fa,$ of $A$ consists of the diagonal
matrices, with diagonal entries $\alpha_i$ satisfying the trace condition 
\begin{equation}
\sum\limits \alpha_i=0.
\label{CSG.217}\end{equation}
The $E_{ij}$ are joint eigenvectors for the adjoint action of $\Fa$ with
collective eigenvalues, the positive roots, being the elements $\alpha
_i-\alpha _j$ of $\Fa'.$ The primitive roots are the $\alpha _i-\alpha
_{j+1}$ for $1\le i\le n-1.$ The dual basis of $\Fa,$ the coroots, are the
diagonal matrices with first $k$ diagonal entries $1-\tfrac kn$ and
remaining entries $-\tfrac kn.$ This decomposes $\Fa$ into a sum of
one-dimensional spaces, with basis elements, and hence decomposes $A$ into a
product of half-lines $A_i$ with coordinates (the coweights) $t_i$ so that
the coroots become $t_i\pa_{t_i}.$ In terms of these coordinates on $A$ an
element of $A_i$ is of the form
\begin{equation}
\diag(t_i^{1-\tfrac kn},\dots, t_i^{1-\tfrac kn},t_i^{-\tfrac
  kn},\dots, t_i^{-\tfrac kn})
\label{CSG.218}\end{equation}
Observe that as $t_i\to\infty$ the projection of this $1$-parameter group
into $\bSL(n,\bbR)$ is smooth up to the boundary in terms of the parameter
$\tau_i=t_i^{-1}$, where it meets $H_{i}$ as a normal vector field with
$z_i$ extending to a local boundary defining function.

The other Borel subgroups $B\subset\SL(n,\bbR)$ are conjugate, under the
action of $\SO(n),$ to this basic choice, so the flag variety is
identified 
\begin{equation}
\cF(\bbR^n,\{1,\dots,n-1\})=\SL(n)/\SO(n),
\label{CSG.219}\end{equation}
and this action conjugates the Iwasawa decomposition as well, giving the
corresponding groups $A_B$ and $N_B$ replete with their root space
decompositions.

From this we conclude that in the decomposition \eqref{CSG.135} near the
`Borel' face, of the maximal codimension $n-1,$ the parameters $\tau_i$
correspond (at least to first order at the boundary) to the action on the
right by $A_B$ or on the left by $A_{B'}$ on the fiber of the boundary
above $(B,B')\in \SL(n)/\SO(n)\times \SL(n)/\SO(n).$ This certainly shows
the b-normality of the extended action since it persists nearby and then
extends to all boundary points using the left and right actions of
$\SO(n).$

The b-transversality also follows directly from this analysis. At the
maximal codimension face the Lie algebra of the right action projects
surjectively to the tangent space of the right flag variety \eqref{CSG.219}
and the left action projects surjectively to the left factor. The b-normal
space is in the range of either action and the fiber is discrete. By
smoothness, the surjectivity of the map from $\bSL(n)\times\Fg\times\Fg$ to
$\bT\bSL(n)$ extends to an open neighborhood. Every orbit of the $G\times
G$ action intersects such an open neighborhood of this face so
b-transitivity extends to the whole space. In fact this is also clear from
the analysis below of the behavior near a general boundary face.

Finally, it remains to show that the minimality condition holds. Again it
suffices to compute the behavior of the Lie algebra for the right action
near one point in the boundary face of maximal codimension and then extend
the result using the action of $G\times G.$ Choosing the base point to be
the standard flag in \eqref{CSG.203} with the $e_i=1$ in this case the
action of the element $E_{ij}\in\Fn,$ $j>1,$ on the right acts on the flag
on the right as $\ha(E_{ij}-E_{ji})\in K.$ The infinitesmal action of $\Fn$
on the right also shifts the left flag in the same way, in this case after
conjugation by $\gamma.$ This the right action of $E_{ij}$ projects on the
left flat 
\begin{equation}
\ha\left(\prod_{k\ge j-1}\tau_k\right)\kappa_{ij},\ \kappa _{ij}=(E_{ij}-E_{ji}),\ j>i.
\label{CSG.220}\end{equation}

It follows that nearby the projection of the right action of the Lie
algebra onto the left flag manifold vanishes at the boundary and there is
an element (with $j=i+1)$ vanishing simply on a given hypersurface. This
the minimality condition is also satisfied.
\end{proof}

From the analysis of the action on the right of the Lie algebra leading to
\eqref{CSG.220} it follows that this compactification has the properties
discussed in the Introduction. Thus on the flag variety corresponding to
Borel subgroups, $\cF=\SO(n)/Z,$ the tangent bundle is stratified by
one-dimensional subbundles $E^\alpha$ where $\alpha$ runs over the
$(n-1)$-multiindices with entries $1$ forming an interval and all others
vanishing. The commutators of these subbundles are contained in the sums
unless the two intervals are non-overlapping but contiguous, in which case
the commutator, modulo the sum, span the bundle corresponding to the union.
  
\section{$\SL(n,\bbC)$}\label{nC}

For the complex group, $\SL(n,\bbC),$ there are two distinct generalization
of the compactification via the resolution of the projective image as
carried out above. In the first we proceed by passing systematically to the
complex category and in the second by proceeding by direct analogy. The
first approach yields the wonderful compactification of the adjoint group
$\SL(n,\bbC)/Z$ where the center is the multiplicative group of $n$th roots
of unity. The second approach gives an hd-compactification.

Consider an Hermitian vector space $V$ of complex dimension $n$ then in place
of \eqref{CSG.171} 
\begin{equation}
\SU(V)\subset\SL(V)\subset\GL(V)\subset\hom(V)
\label{CSG.190}\end{equation}
where the last three spaces are complex. Let
$\Ph(V)=(\hom(V)\setminus\{0\})/\bbC^\times$ be the projective space of
homomorphisms. Projection gives 
\begin{equation}
\SL(V)\longrightarrow \PI(V)\subset\Ph(V)
\label{CSG.191}\end{equation}
where the center is mapped to the image of the identity and the range,
$\PI(V)$ is an open dense subset of $\Ph(V).$ The complement
$\Ph(V)\setminus\PI(V)$ is the image of the non-vanishing, non-invertible,
matrices and so is again stratified by corank, 
\begin{equation}
\Ph(V)\setminus\PI(V)=\bigcup_{1\le k<n} S_k.
\label{CSG.192}\end{equation}
These are all complex submanifolds, although only $S_{n-1}$ is compact. As
in the real case they are the isotropy types for the action of $\SU(V)/Z$
which acts freely on $\PI(V).$

Again in this case we may blow these submanifolds up, iteratively, in the
complex sense, in order of increasing dimension to produce the resolved
space which we denote tentatively as
\begin{equation}
\bPh_{\DP}(V)=[\Ph(V);S_*]_{\bbC}.
\label{CSG.198}\end{equation}

To see that this is well-defined as a compact complex manifold note that an element
of $S_{n-1}$ is the projective image of a complex homomorphism of rank one, so
\begin{equation}
S_{n-1}=\bbP V\times\bbP V
\label{CSG.193}\end{equation}
where the first point is the orthcomplement of the null space and the
second the range. The normal bundle to $S_{n-1}$ can then be identified
with the bundle with fibre at a point the homomorphisms from the null space
to the orthocomplement. Blowing up replaces $S_{n-1}$ be its projectivized
normal bundle $D'_{n-1}$  
\begin{equation}
\xymatrix{
\Ph(\bbC^{n-1})\ar@{-}[r]& D'_{n-1}\ar[d]\\
&\bbP V\times\bbP V.
}
\label{CSG.194}\end{equation}
In a small neighbourhood of this exceptional divisor the resolved manifold
is a bundle over $\bbP V\times\bbP V$ with fiber a neighborhood of the zero
section of the blow-up of the origin in the homomorphism bundle from null
space to orthocomplement of range. The closures of the other strata meet
this bundle in the corresponding strata modeled on $S_j(\Ph(\bbC^{n-1})).$
Proceeding by induction as in the real case these lower depth strata can be
blown up in order and replace the divisor \eqref{CSG.194} by its resolution
giving the bundle 
\begin{equation}
\xymatrix{
\bPh(\bbC^{n-1})\ar@{-}[r]& D_{n-1}\ar[d]\\
&\bbP V\times\bbP V.
}
\label{CSG.195}\end{equation}

It follows that the full resolution is possible with the stratum $S_j$
replaced by a divisor which is a complex bundle over the product of two
copies of the Grassmannian
\begin{equation}
\xymatrix{
\bPh(\bbC^{j})\times \bPh(\bbC^{n-j})\ar@{-}[r]& D_{j}\ar[d]\\
&\Gr(V,n-j)\times\Gr(V,n-j).
}
\label{CSG.196}\end{equation}
\begin{proposition}\label{CSG.197} This resolution, $\bPh(V),$ of $\Ph(V)$ is the
  wonderful compactification of the adjoint group $\SL(n,\bbC)/Z.$
\end{proposition}

This is consequence of the characterization of the wonderful
compactification as the unique regular compactification with a single
compact orbit, see Uma \cite{1805.02135} and Brion \cite{Brion1989} and
references therein.

Real and complex blow-up of a complex submanifold are related in a simple
way, namely the real blow-up is the complex blow-up followed by the real
blow up of the resulting divisor, as a real submanifold of codimension
two. For normally intersecting divisors this carries over iteratively --
the real blow-ups can all be performed after the complex ones because of transversality.

Now Prop-W in the Introduction can be stated more explicitly.

\begin{proposition}\label{CSG.200} The blow-up, in the real sense, all of
  the divisors in \eqref{CSG.198} gives an hd-compactification of the
  adjoint group $\SL(n,\bbC)/Z:$
\begin{equation}
\overline{\SL(n,\bbC)/Z}=[\bPh V;S_*]_{\bbR}.
\label{CSG.199}\end{equation}
\end{proposition}

Rather than prove this immediately we examine the compactification of
$\SL(n,\bbC)$ obtained by following the procedure for $\SL(n,\bbR)$ more
closely. That is, first project radially into the sphere in the Hermitian vector
space  
\begin{equation}
\SL(V;\bbC)\longrightarrow \Sh(V).
\label{CSG.201}\end{equation}
Again this is a diffeomorphism but now onto the smooth hypersurface, $\Sh_+(V),$ in
$\Sh(V)$ where the determinant is positive.

The actions of $\SU(V)$ on $\Sh(V),$ left and right, fix $\Sh_+(V)$ and act
freely on it. The isotropy types for the action on $\Sh(V)$ are again the
submanifolds $S_q\subset\Sh(V)$ of corank $q.$ We define a resolution of
$\SL(V)$ by taking the closure in the (real) resolution 
\begin{equation}
\bSL(V)\subset[\Sh(V),S_*].
\label{CSG.221}\end{equation}

The determinant is well-defined and non-zero on $\SI(V)\subset\SH(V)$ and
we define a normalized version by
\begin{equation}
\widehat{\det}:\SI(V)\ni g\longmapsto \frac{\det g}{|\det g|}\in\bbT.
\label{CSG.137}\end{equation}

Now the $\SU(V)$ actions extend to actions of $\UU(V)$ with the
extended actions free on $\SI(V).$ The isotropy types are the same as for
the $\SU(V)$ action, i.e. the $S_q.$ Thus in the case $W_1=W_2=V$ the
resolution \eqref{CSG.221} also resolves the $\UU(V)$ actions, which
therefore become free on $\bSH(V).$ The normalized determinant factors
through the action of $\UU(V),$
\begin{equation}
\widehat{\det}(ug)=\det(u)\widehat{\det}(g)
\label{CSG.138}\end{equation}
and it follows from this that $\widehat{\det}$ extends smoothly to
$\bSH(V)$ and has non-vanishing differential everywhere. In fact the
differential must be independent of all conormals as well, so the level
surfaces are well-defined p-submanifolds of $\bSH(V).$

Thus the definition \eqref{CSG.221} does lead to a compact manifold with
corners, with extended actions of $\SL(V),$ since the initial action
preserves the $S_q.$

That this is an hd-compactification now follows as in the real case
discussed above. Moreover there is a simple relation between the real and
complex cases:

\begin{proposition}\label{CSG.145} The closure of $\SL(n,\bbR)$ in the
  compactification $\bSL(n,\bbC)$ is a p-submanifold which is an
  hd-compactification. 
\end{proposition}
\section{Direct construction}\label{DC}

In the construction of an hd-compactification of $G,$ a real reductive
group with compact center, we use the root space decomposition, see for
example \cite{MR1880691}, which we proceed to summarize briefly. First we
make the choice of a `real Cartan subgroup' of $G,$ a maximal product of
multiplicative half-lines, $A$ with Lie algebra $\Fa$ and
compatible Cartan involution. The adjoint action of $\Fa$ 
on the full Lie algebra, $\Fg,$ is symmetric and decomposes $\Fg$ into the
joint null space and a direct sum of eigenspaces with
corresponding joint eigenvectors, the roots, in $\Fa^*.$ The choice of a
positive subspace, $R_+,$ of the roots induces an Iwasawa decomposition
\begin{equation}
G=KAN
\label{CSG.176}\end{equation}
where $K$ is a maximal compact subgroup and $N$ is the unipotent group
generated by the span of the strictly positive root (eigen)vectors. This also
induces a decomposition 
\begin{equation}
A=\prod_{i\in D}^rA_i,\ \tau_i:A_i\simeq \bbR^+,\ \#(D)=r,
\label{CSG.177}\end{equation}
as a product of half-lines labeled by the nodes, $D,$ of the Dynkin
diagram. The parameters in the factors which may be identified with the
coweights. Thus $\Fa$ is spanned by the $\tau_i\pa_{\tau_i}$ and each of
the root vectors spanning the Lie algebra $\mathfrak{n}$ of $N$ has
non-negative homogeneity under the adjoint action $A_i$ (with at least one
positive exponent). This decomposes $\Fn$ as a graded algebra under the
conjugation action of the $A_i;$
\begin{equation}%
\begin{gathered}
\Ad(a)n =\tau^{\alpha }n,\ n\in
\mathfrak{n}_\alpha,\ a\in A,\ \Fn=\bigoplus_{\alpha\in R_+} \mathfrak{n}_\alpha\\
\mathfrak{n}_{\alpha +\beta }=[\Fn_\alpha ,\Fn_\beta ]\Mif \alpha ,\ \beta
,\ \alpha +\beta \in R_+ .
\end{gathered}
\label{CSG.178}\end{equation}

These choices correspond to a particular minimal parabolic subgroup 
\begin{equation}
P=MAN
\label{CSG.179}\end{equation}
where in this case the real-reductive part is compact.  The
other minimal parabolics are conjugate to this one, under the action of $K,$
and so have a similar  decomposition \eqref{CSG.178} and \eqref{CSG.179}.

Each conjugacy classes of parabolic sugroups correspond to a subset $S\subset
D$ where the minimal case corresponds to the empty set. To $S$ we
associate the group
\begin{equation}
A_S=\prod_{i\notin S}A_i\subset A.
\label{CSG.180}\end{equation}
The normalizer of $A_S$ in $G$ contains a maximal real-reductive subgroup
$M_S$ with compact center fixed by the Cartan involution. The corresponding
parabolic subgroup has Langlands decomposition
\begin{equation}
P_S=M_SA_SN_S
\label{CSG.182}\end{equation}
where the unipotent group $N_S$ has Lie algebra 
\begin{equation}
\mathfrak{n}_S=\bigoplus_{\gamma \in R^S_+}n_\gamma .
\label{CSG.181}\end{equation}
Here $R^S_+$ is the subset of the positive roots which have a positive
value on one of the basis elements of the Lie algebra of $A_S.$ This induces an action
of $A_S$ of the form \eqref{CSG.178} for these restricted roots.

Again the conjugation action of the maximal compact subgroup in the initial
Iwasawa decomposition induces similar decompositions of all the conjugate
parabolics $P=M_PA_PN_P$ independent of the indeterminacy in conjugation.

Let $\cF_S$ be the flag variety of parabolic subgroups conjugate to
$P_S.$ For a pair $(P',P)\in\cF_S\times \cF_S$ consider the space
\begin{equation}
F(P,P')=\{g\in K\cdot M_P;M_{P'}g=gM_P\}.
\label{29.8.2019.1}\end{equation}

\begin{lemma}\label{29.8.2019.4} For each $S\subset D,$ the spaces $F(P',P)$ have
principal $M_P$ actions on the right and $M_{P'}$ actions  on the left and
form a smooth bundle with total space $F_S$
\begin{equation}
\xymatrix{
M_S\ar@{-}[r]&F_S\ar[d]\\
&\cF_S\times\cF_S
}
\label{CSG.161a}\end{equation}
and fiber modeled on $M_S.$
\end{lemma}

\begin{proof} The action of $K$ on $\cF_S,$ by conjugation on $G,$ is
  transitive with isotropy group at $P\in\cF_S$ the subgroup $K\cap M_P.$
  It follows that if $k\in K$ and $kPk^{-1}=P'$ then $F(P,P')=kM_P$ from
  which the result follows.
\end{proof}

The properties of the putative compactification functor imply that it
extends to principal spaces such as $F(P,P')$ so, proceeding inductively
over dimension or real rank, we can define the \emph{closed} boundary face
corresponding to each conjugacy class of parabolics in $G$ as the fiberwise
compactification of the smooth bundle \eqref{CSG.161a} with fibre the
compactification of $F(P',P)\simeq M_S$
\begin{equation}
\xymatrix{
\bM_S\ar@{-}[r]&\bF_S\ar[d]\\
&\cF_S\times\cF_S.
}
\label{29.8.2019.5}\end{equation}

We define a partial compactification
\begin{equation}
A_S\hookrightarrow \bA_S
\label{CSG.183}\end{equation}
of each $A_S$ by adding a point at infinity to each factor, with smooth
structure given by the parameter $\tau_i^{-1}$ where $\tau_i$ is the
multiplicative variable in \eqref{CSG.177}. Thus $\bA_S$ is a product of
$\#(D\setminus S)$ half-closed intervals. The conjugation action of $K$
induces corresponding compactifications $\bA_P$ of each $A_P$ and so
defines a bundle over $\cF_S$ and hence another level on the fiber bundles
\eqref{29.8.2019.5}:
\begin{equation}
\xymatrix{
\bA_S\ar@{-}[r]&N_+\bF_S\ar[d]\\
\bM_S\ar@{-}[r]&\bF_S\ar[d]\\
&\cF_S\times\cF_S.
}
\label{CSG.184}\end{equation}
The fibre of $N_+\bF_S$ is $\bA_P$ over $(P,P')$ but may also be naturally
identified with $\bA_{P'}.$ The presumptuous notation $N_+\bF_S$ presages the
identification with the inward-point normal bundle to $\bF_S$ as a boundary
face of $\bG.$

To construct $\bG$ as a smooth space we define gluing maps, the most
fundamental one (recalling that we proceed by induction) is for the minimal
face corresponding to \eqref{CSG.184} with $S=\emptyset;$ in this case we
drop the subscript. Then $M$ is already compact. The partially compactified
fibre at each point $(B,B')\in\cF\times\cF$ has a point of maximal
codimension, corresponding to passage to infinity in all factors
$A_i\subset A.$ This defines a section, realizing the zero section of
$N_+F;$ the bundle $F$ in this case is already compact.

\begin{proposition}\label{CSG.187} The map 
\begin{equation}
\gamma :N_+F\ni (m,a)\longrightarrow ma\in G,\ m\in F(P,P'),\ (P',P)\in\cF\times\cF
\label{CSG.185}\end{equation}
is smooth and if $\bO\subset N_+F$ is a small (relatively) open
neighborhood of the zero section of $N_+F,$ corresponding to
minimal parabolic subgroups, then $\gamma$ is a diffeomorphism of
$O,$ the intersection of $\bO$ with the interior
of $N_+F$ as a manifold with corners, onto an open subset of $G.$
\end{proposition}
\noindent In fact we can take $\bO$ to be invariant under the left and
right actions of $K.$

\begin{proof} The map is well-defined and smooth. In the discussion above
  we have chosen a base minimal parabolic and this induces an
  identification $\cF=K/Z.$ The total space of the bundle $F,$ the boundary
  face here is $K\times_Z K$ and the map becomes
\begin{equation}
g=k_1ak_2,\ k_i\in K,\ a\in A
\label{CSG.188}\end{equation}
which descends to $K\times_Z K\times A.$ This is the Cartan
decomposition. In the open set $O,$ where all the components of $a$ are
large, the decomposition is unique, up to the indeterminacy in the
center. The factors may be chosen smoothly, locally, which shows
\eqref{CSG.185} to be a smooth bijection onto its open image with smooth
local inverses, i.e.\ a diffeomorphism.
\end{proof}

To complete the construction of $\bG$ we extend the map \eqref{CSG.188} to
the bundles corresponding to the other parabolics and then to similar maps
between these bundles. For each $S\subset D$ the total space of $F_S$ may
be identified with
\begin{equation}
F_S\simeq K\times_{K_S}M_S\times_{K_S}K,\ K_S=K\cap M_S, P_S=M_SA_SN_S.
\label{CSG.205}\end{equation}
Then consider the smooth map 
\begin{equation}
\gamma_S:K\times M_S\times K\times  A_S\ni(k_1,m,k_2,a)\longmapsto k_1mak_2\in G.
\label{CSG.206}\end{equation}
This descends to $F_S,$ using \eqref{CSG.205}.

\begin{theorem}\label{29.8.2019.6} The total spaces of the fibrations \eqref{29.8.2019.5}
form the boundary stratification of an hd-compactification 
\begin{equation}
\overline{G}=G\sqcup \bigcup_{S\subsetneq D}\bF_S
\label{29.8.2019.7}\end{equation}
where the \ci\ structure, near each boundary face, is fixed by the gluing maps
\eqref{CSG.206}.
\end{theorem}


\begin{proof} As noted above,
  we proceed by induction over the dimension of $G,$ allowing for
  real-reductive groups with compact centers. In particular this means that
  we take as given the compactification in \eqref{29.8.2019.5}, and hence
  all the $\bF_S$ are well-defined when we come to construct $\bG.$

The partial order on the $S\subset D$ by inclusion induces a partial order
on the $\bF_S$ which corresponds to inclusion as boundary faces. Consider
two subsets $S_1\subsetneq S_2\subsetneq D$ -- so excluding the interior
which corresponds to  $D.$ Set $M_i=M_{S_i}$ etc, so in particular
$K_i=K\cap M_i.$ The partially compactified `normal' spaces are
products
\begin{equation}
\bA_{1}=\bA_{2}\times\bA_{12}.
\label{CSG.207}\end{equation}
There is a corresponding `lifted' gluing map
\begin{multline}
\gamma_{12}:
K\times K_2\times M_1\times K_2\times K \times A_1\ni(k',k'_1, m,k_1'', k'',a_2a_{12})\\
\longmapsto
(k',k'_1ma_{12}k_1'',k'',a_2)
\label{CSG.208}\end{multline}
which descends to a map 
\begin{equation}
N_+F_1=K\times_{K_1}M_1\times_{K_1}K\times A_1\longrightarrow
K\times_{K_2}M_2\times_{K_2}\times A_2=N_+F_2.
\label{CSG.209}\end{equation}
Since these are proper boundary faces, corresponding to lower dimensional
groups $M_S,$ the inductive hypothesis means that the maps extend to smooth
maps on the compactified spaces
\begin{equation}
\gamma_{12}:N_+\bF_1\longrightarrow N_+\bF_2.
\label{CSG.212}\end{equation}
Moreover the gluing maps \eqref{CSG.206}, $\gamma_i$ to $G$ for the two
boundary faces factor
\begin{equation}
\gamma_2=\gamma_1\circ\gamma_{12}.
\label{CSG.210}\end{equation}

If $U_S\subset M_S$ is an open subset with compact closure and $\bO_S$ is a
sufficiently small, relatively open neighborhood of the point at infinity
in $\bA_S,$ corresponding to the zero section of the normal bundle, then
the map $\gamma_S$ restricted to the image of the set $K\times U_S\times
K\times O_S$ is a diffeomorphism onto its range in $G.$

Now, starting from the face of maximal codimension, we can successively
choose such subsets which together cover the union of the boundary faces in
the sense that for all $S$
\begin{equation}
M_S\subset U_S\cup \bigcup_{S'\subsetneq S}\gamma_{SS'}(K\times
M_{S'}K\times K\times O_{SS'}).
\label{CSG.211}\end{equation}

These choices give the \ci\ structure on $\bG.$ Namely all of the maps
$\gamma_S$ and $\gamma_{SS'}$ are diffeomorphisms when restricted to these
small domains. The conditions \eqref{CSG.210} mean that the
maps on the components of the covering of $M_S$ in \eqref{CSG.211},
$\gamma_S$ on $U_S$ and the composite map $\gamma_{S'}\gamma_{SS'}^{-1}$ on
the other parts are consistent on overlaps. Each of these maps identifies
an open neighborhood of the boundary face $F_S$ with a corresponding open
subset of the inward-pointing normal bundle $N_+F_S$ with smoothness on the overlaps, thus
making $\bG$ into a manifold with corners.

To see that $\bG$ is an hd-compactification of $G$ we check the
conditions in Definition~\ref{CSG.167}. First observe the effect of
inversion on the image of one of the gluing maps 
\begin{equation}
(k_1mak_2)^{-1}=k_2^{-1}m^{-1}a'k_1^{-1}.
\label{CSG.213}\end{equation}
Here $(P,P')$ are reversed, $m^{-1}\in F(P',P)=F(P'_-,P_-)$ and $a'$ is a
point near infinity in $A_{P'_-}$ where $P_-$ is the opposite parabolic,
which corresponds to inverting $A_P$ and changing to the negatives of the roots.
So this does indeed extend to a diffeomorphism.

Next consider the right action of $g\in G$ on the image of some $\gamma_S.$
The action of $K$ extends smoothly (and freely) up to the boundary by
conjugating the 'incoming' parabolic $P$ to $P''$ and over this action on
$\cF_S$ mapping $A_P$ to $A_{P''}$ and $F(P,P')$ to $F(P'',P').$ So it
suffices to consider the action of $P$ near a boundary fiber over $(P',P).$
In the Langlands decomposition the action of $M_P$ on $F(P,P')$ is smooth,
free and transitive and near the identity in $A_P$ the action factors
through that on $\bA_P,$ in particular fixing the boundary. So it remains
to check the action of the unipotent group $N_P.$ Since $a\in A_P$ acts on
$N_P$ by conjugation, as in \eqref{CSG.178}, 
\begin{equation}
an=n_aa,\ n_a\text{ smooth on }\bA_P  
\label{CSG.214}\end{equation}
and with $n_a$ vanishing when $a$ is at the point at infinity. Thus $N_P$
fixes the boundary fibers above $(P,P')$ for all $P'.$ The smoothness of
the action follows. The b-normality of the action follows from the fact
that at a boundary fiber above $(P,P')$ the action of $A_P$ is through
scaling on the normal fiber.

This discussion shows that over the boundary face $F_S\longrightarrow
\cF_S\times\cF_S$ the right action fixes the left factor of $\cF_S$ in the
base and acts transitively on the fibers.  Since inversion conjugates the
right to left action and reverses the factors in the base it follows that
the action of $G\times G$ is b-transitive.

Finally, that the mininality condition holds can be seen from the local
homogeneity in \eqref{CSG.178}.
\end{proof}

\providecommand{\bysame}{\leavevmode\hbox to3em{\hrulefill}\thinspace}
\providecommand{\MR}{\relax\ifhmode\unskip\space\fi MR }
\providecommand{\MRhref}[2]{%
  \href{http://www.ams.org/mathscinet-getitem?mr=#1}{#2}
}
\providecommand{\href}[2]{#2}

\end{document}